\documentclass[11pt,a4paper,reqno]{amsart}
\usepackage[utf8]{inputenc}

\usepackage[a4paper,lmargin=3.5cm,rmargin=3.5cm,tmargin=4cm,bmargin=4cm]{geometry}

\usepackage[centertags]{amsmath}
\usepackage{amsmath,amsthm,amssymb,amsfonts}
\usepackage{multicol}

\usepackage[bookmarks=true,hyperindex,pdftex,colorlinks,linkcolor=Blue, citecolor=Blue, urlcolor=Blue,linktocpage=true]{hyperref}

\usepackage[pdftex,svgnames,dvipsnames]{xcolor}

\makeatletter
\g@addto@macro\bfseries{\boldmath} 
\makeatother

\usepackage[normalem]{ulem}
\usepackage{bbm}
\usepackage{mathabx}
\usepackage{enumitem}
\setenumerate[1]{label=(\roman{*}), ref=(\roman{*})}
\setenumerate[2]{label=(\alph{*}), ref=(\alph{*})}
\SetLabelAlign{Center}{\hfil#1\hfil}
\usepackage{cancel}
\usepackage{ulem}
\usepackage{indentfirst}
\usepackage{ragged2e}

\newcommand{\K}{\mathbb K}

\newcommand{\C}{\mathbb C}
\newcommand{\N}{\mathbb N}

\newcommand{\R}{\mathbb R}

\newcommand{\abs}[1]{\left\vert#1\right\vert}

\newcommand{\F}{\mathcal F}

\newcommand{\eps}{\varepsilon}

\newcommand{\Pol}{\mathcal P}

\DeclareMathOperator{\re}{Re}
\DeclareMathOperator{\OF}{OF}

\newcommand{\norm}[1]{\left\Vert#1\right\Vert}

\newcommand{\scal}[1]{\left\langle#1\right\rangle}

\newcommand{\Lip}{\mathop\mathrm{Lip}}

\DeclareMathOperator{\real}{Re}

\newcommand{\Id}{{\mathrm{Id}}}
\DeclareMathOperator{\dens}{dens}

\newcommand{\lipfree}[1]{\mathcal{F}({#1})} 

\theoremstyle{plain}
\newtheorem*{thm*}{Theorem}
\newtheorem*{prop*}{Proposition}
\newtheorem*{cor*}{Corollary}
\newtheorem{thm}{Theorem}[section]
\newtheorem{cor}[thm]{Corollary}
\newtheorem{lem}[thm]{Lemma}
\newtheorem{prop}[thm]{Proposition}

\theoremstyle{definition}
\newtheorem{defn}[thm]{Definition}
\newtheorem{rem}[thm]{Remark}

\begin{document}

\title[The polynomial Daugavet property]{The Daugavet property is equivalent to the polynomial Daugavet property}

\date{August 14th, 2025; Revised July 14th, 2026}

\author[S.~Dantas]{Sheldon Dantas}
\address[S.~Dantas]{Czech Technical University in Prague, FEE, Department of Mathematics, Technick\'{a} 2, 16627, Prague 6, Czech Republic \newline
\href{https://orcid.org/0000-0001-8117-3760}{ORCID: \texttt{0000-0001-8117-3760}}}
\email{\texttt{sheldon.dantas@fel.cvut.cz}}
\urladdr{www.sheldondantas.com}

\author[Mart\'in]{Miguel Mart\'in}
\address[Mart\'in]{Department of Mathematical Analysis and Institute of Mathematics (IMAG), University of Granada, E-18071 Granada, Spain}
\email{mmartins@ugr.es}
\urladdr{\url{https://www.ugr.es/local/mmartins/}}
\urladdr{
\href{http://orcid.org/0000-0003-4502-798X}{ORCID: \texttt{0000-0003-4502-798X} } }

\author[Perreau]{Yo\"el Perreau}
\address[Perreau]{Institute of Mathematics and Statistics, University of Tartu, Narva mnt 18, 51009 Tartu, Estonia}
\email{yoel.perreau@ut.ee}
\urladdr{
\href{https://orcid.org/0000-0002-2609-5509}{ORCID: \texttt{0000-0002-2609-5509} } }

\begin{abstract} 
In this note, we prove that the Daugavet property implies the polynomial Daugavet property, solving a longstanding open problem in the field. Our approach is based on showing that a geometric characterization of the Daugavet property due to Shvidkoy, originally formulated in terms of the weak topology, remains valid when the weak topology is replaced by the weak polynomial topology. Using similar techniques, we further establish that every linear Daugavet center is also a polynomial Daugavet center, and that the weak operator Daugavet property implies its polynomial counterpart. As an application of the latter result, we present new examples of Banach spaces whose $N$-fold symmetric tensor products satisfy the Daugavet property.
\end{abstract}

\maketitle

\section{Introduction}\label{sec:introduction}
Given a Banach space $X$ over a field $\K$ ($\K=\R$ or $\K=\C$), we denote by $B_X$ its closed unit ball and by $S_X$ its closed unit sphere. We also denote by $X^*$ the topological dual of $X$ and by $X^{**}$ its bidual. For Banach spaces $X$ and $Y$, we write $\mathcal{L}(X,Y)$ to denote the Banach space of all bounded linear operators from $X$ to $Y$ endowed with the operator norm. The symbol $\Id$ stands for the identity operator. We will write $\re z$ to denote the real part of $z\in \K$; of course, when $\K=\R$, this simply means $\re z=z$.

A Banach space $X$ has the \emph{Daugavet property} (\emph{DPr}, for short), if the \emph{Daugavet equation}
\begin{equation}\tag{DE}\label{eq:DE}
\norm{\Id + T} = 1+\norm{T}
\end{equation}
holds for every rank-one $T\in\mathcal{L}(X,X)$. In this case, we have that every weakly compact operator $T\in \mathcal{L}(X,X)$ also satisfies \eqref{eq:DE} (in fact, many others classes of linear operators satisfy \eqref{eq:DE}, see the survey \cite{Werner01} for more on the topic). This property takes its name after I.~K.~Daugavet's paper from 1963, in which it is proved that compact linear operators on $C[0,1]$ satisfy \eqref{eq:DE}. 
Examples of Banach spaces with the DPr include $C(K)$-spaces when $K$ is perfect \cite{Daugavet63}, $L_1(\mu)$-spaces when $\mu$ has no atoms \cite{Lozanovskii66}, Lipschitz-free spaces $\lipfree{M}$ and Lipschitz spaces $\Lip(M)$ when $M$ is a length metric space \cite{IKW07,GLPRZ18}, non-atomic $C^*$-algebras and preduals of diffuse von Neumann algebras \cite{BM05,Oik02}, and some Banach algebras of holomorphic functions on an arbitrary Banach space \cite{Jung23,Werner97,Wojtaszcyk92}, among many others. It is immediate that the DPr passes from the dual to the space itself, while the other implication is false, as the above list of examples shows. We refer the reader to the recent monograph \cite{KMRZW25} for a detailed account and background on the Daugavet property. Considerable attention has been devoted to this isometric property of Banach spaces, which has several notable consequences on the isomorphic structure of the underlying Banach space. For instance, a Banach space with the Daugavet property cannot be isomorphically embedded into a Banach space with an unconditional Schauder basis (see \cite{KSSW00}), extending a classical result by A.~Pe{\l}czy\'{n}ski for $L_1[0,1]$. Other important isomorphic consequences of the DPr on a Banach space are that it cannot satisfy the Radon-Nikod\'{y}m property and that it contains (many) isomorphic copies of $\ell_1$ (hence, it cannot be Asplund). These latter consequences can be proved by using a celebrated geometric characterization of the DPr in terms of slices of the unit ball which was first obtained in \cite{KSSW00}. This result can be restated in the following way.

\begin{prop}[\mbox{\cite[Lemma 2.2]{KSSW00}}]\label{prop:DP_slice_characterization}
   A Banach space $X$ has the DPr if and only if its unit ball $B_X$ is equal to the closed convex hull of the set 
   \begin{equation*}
       \{y\in B_X\colon \norm{x+y}>2-\eps\}
   \end{equation*} 
   for every $x\in S_X$ and $\eps>0$. In other words, for every $x\in S_X$ and every slice $S$ of $B_X$, we have 
    \begin{equation*}
        \sup_{y\in S}\norm{x+y}=2
    \end{equation*}
\end{prop}
\noindent
(recall that a \emph{slice} of a subset $C$ of a Banach space $X$ is a non-empty intersection of $C$ with an open half-space).

The above characterization was further refined by R.~Shvidkoy in \cite{Shvidkoy00}, after he rediscovered a geometric lemma due to J.~Bourgain concerning convex combinations of slices. In particular, the following result was established, which we shall refer to as {\it Shvidkoy's Lemma}.

\begin{lem}[\mbox{Shvidkoy's lemma \cite[Lemma~3]{Shvidkoy00}}] \label{lemma:DP_Shvidkoy_characterization}
Let $X$ be a Banach space with the DPr. Then, for every $x\in S_X$ and $\eps>0$, the set $\{y\in B_X\colon \norm{x+y}>2-\eps\}$ is weakly dense in $B_X$. In other words, for every $x\in S_X$ and $y\in B_X$, we can find a net $(y_\alpha)$ in $B_X$ such that $y_\alpha\to y$ weakly and $\norm{x+y_\alpha}\to2$.  
\end{lem}

Following the terminology of \cite[\S~2.6]{KMRZW25}, two elements $x,y\in B_X$ are called \emph{quasi-codirected} if they satisfy $\|x+y\|=2$; given $\eps>0$, we say that $x,y$ are  \emph{$\eps$-quasi-codirected} if they satisfy $\|x+y\|>2-\eps$.

As a consequence of Shvidkoy's lemma, it can be proved that a Banach space $X$ with the DPr cannot satisfy the CPCP (in fact, every relatively weakly open subset of $B_X$ has diameter 2) and that, in the separable case, its bidual contains many $L$-orthogonal elements (there is a weak-star $G_\delta$-dense subset $\mathcal{A}$ of $B_{X^{**}}$ such that $\|x+x^{**}\|=\|x\|+1$ for every $x\in X$ and every $x^{**}\in \mathcal{A}$, see \cite[Corollary 3.1.18]{KMRZW25} for this version and see \cite{AMCRZ23,LRZ2021} for further results in this line). Lemma~\ref{lemma:DP_Shvidkoy_characterization} can be also used to show that operators not fixing copies of $\ell_1$ on a Banach space with the DPr satisfy the Daugavet equation \cite[Theorem~4]{Shvidkoy00}.  

A natural extension of the DPr to polynomials was first introduced in \cite{CGMM07}, while the name appeared in the later paper \cite{CGMM08} (see Subsection~\ref{subsec:preliminaries} below for the necessary background on polynomials on Banach spaces). A Banach space $X$ has the \emph{polynomial Daugavet property} (\emph{polynomial DPr}, for short) if every weakly compact polynomial $P\colon X\to X$ satisfies the equation
\begin{equation*}
    \norm{\Id+P}=1+\norm{P}.
\end{equation*}
This definition is in fact equivalent to requiring the same just for rank-one polynomials \cite[Theorem~1.1]{CGMM07}.
As bounded linear operators are, in particular, polynomials, the polynomial DPr implies the DPr. 
In that same paper \cite{CGMM07}, the following geometric characterization of the polynomial DPr was provided.

\begin{prop}[\mbox{\cite[Corollary 2.2]{CGMM07}}]\label{prop:char-pol-DPr}
    A Banach space $X$ has the polynomial DPr if and only if it satisfies the following condition: given $x\in S_X$, $\eps>0$, and a norm-one scalar-valued polynomial $p\in \Pol(X)$, we can find $y\in B_X$ and a modulus-one scalar $\omega\in\K$ such that $\real \omega p(y)>1-\eps$ and $\norm{x+\omega y}>2-\eps$. 
\end{prop}

It was proved in \cite{CGMM07} that the polynomial DPr coincides with the DPr in $C(K)$-spaces and in some spaces of vector-valued functions, and it was observed that a closely related property, the alternative polynomial Daugavet property, differs from its linear version.
The question of whether or not the polynomial DPr and the DPr always coincide, while implicit in \cite{CGMM07,CGMM08}, was first explicitly posed in \cite[Problem 1.2]{MMP10}. In that paper, the authors provide a positive answer in the case of $L_1(\mu)$-spaces and their vector-valued counterparts.  This question was since investigated in various settings, and positive answers was obtained for the following classes of Banach spaces: vector-valued uniform algebras in \cite{CGKM14}, representable Banach spaces in \cite{BS16}, $L_1$-preduals, spaces of Lipschitz-functions, and nicely embedded subspaces of some $C_b(\Omega)$-spaces in \cite{MRZ22}, $C^*$-algebras and JB$^*$-triples with some partial results in \cite{Santos14} and a full solution in \cite{CMP24}. We refer to the recent survey \cite{DGMS23} for an account in this direction.

As announced, the main goal of this paper is to provide a complete positive solution to this problem - namely, that the Daugavet property is equivalent to the polynomial Daugavet property.

To motivate our approach, let us observe that, in view of the above geometric characterization of the polynomial DPr, it is clear that Shvidkoy's Lemma immediately gives the polynomial DPr from the classical DPr provided that the polynomials we are working with are weakly continuous on $B_X$, see \cite[Theorem~2.3]{CMP24}. However, it is known that in infinite dimensional Banach spaces, the only weakly continuous homogeneous polynomials are those of finite type, i.e.\ of the form $x\mapsto \sum_{j=1}^m\alpha_j\varphi_j^N(x)$ with $\alpha_j\in\K$ and $\varphi_j\in X^*$ (see e.g.\ the short paragraph following \cite[Proposition~2.4]{AP80}). Furthermore, it follows from the results in \cite{FGGL83} (see also \cite[Proposition 2.36 and its proof]{Din99}) that in any Banach space that contains a copy of $\ell_1$ (hence, in particular, in any space with the Daugavet property), there exist polynomials which are weakly sequentially continuous but not weakly continuous on bounded subsets. Therefore, a complete answer to our problem cannot be expected as a straightforward consequence of Shvidkoy’s Lemma (see \cite[Section~2]{CMP24} for a more detailed discussion).

A general strategy to deal with the polynomial DPr has been thus to work with sequences instead of nets. In general, polynomials need not be weakly sequentially continuous (consider e.g.\ the polynomial $x\mapsto \scal{x,x}$ on a Hilbert space), but it is well known that if a Banach space $X$ has the Dunford--Pettis property, then all polynomials on $X$ are weakly sequentially continuous, see \cite[Proposition 2.34]{Din99}. In essence, for all the previously mentioned examples, the polynomial DPr was obtained by first constructing suitable approximating sequences in the considered space or its bidual (using Rademacher sequences for $L_1(\mu)$ spaces as in \cite{MMP10}, or constructing some $c_0$ sequences in the space as in \cite{CGMM07,CGMM08} or in the bidual space as in \cite{MRZ22}) and then, by using the weak sequential continuity of polynomials (or their Aron--Berner extensions) for these sequences. One notable exception is the paper \cite{CMP24}, where a completely different topology (more precisely, the strong$^*$ topology) was used to reach the desired conclusion. However, observe that the key point there was again to show some sequential continuity for polynomials under this specific topology (see \cite[Corollary 3.7]{CMP24}). 

While these sequential approaches are interesting by themselves and may bring many information on the considered spaces, there is again little hope that they would provide a complete solution to the polynomial DPr problem, since there is a known example of a Banach space with the DPr and the Schur property \cite{KW04}.

Our approach will follow a different route from those previously explored. Indeed, we will be working directly with the natural topology on the unit ball induced by polynomials, the weak polynomial topology (see Subsection \ref{subsec:preliminaries} for the precise definition).
We will use ``locally'' the $\eps$-quasi-codirected point  argument from Shvidkoy's Lemma \ref{lemma:DP_Shvidkoy_characterization} while revisiting the construction of approximating nets for this topology used by A.M.~Davie and T.W.~Gamelin in the proof of \cite[Theorem~2]{DG89} to produce a Goldstine-type result for the polynomial-star topology of the bidual space. As an outcome of this, we get our main result which is the following version of Shvidkoy's Lemma for the weak polynomial topology.

\begin{thm*}[\mbox{\textrm Theorem~\ref{thm:polynomial_Shvidkoy_result}}]
    Let $X$ be a Banach space with the DPr. Then, for every $x \in S_X$ and for every $\eps>0$, the set $\{y\in B_X\colon \norm{x+y}>2-\eps\}$ is dense in $B_X$ for the weak polynomial topology of $X$.
\end{thm*}

As a consequence, we will obtain that the polynomial DPr and the DPr are equivalent properties.

\begin{cor*}[\mbox{\textrm Corollary~\ref{DPr-polynomialDPr}}]
A Banach space $X$ satisfies the DPr if and only if it satisfies the polynomial DPr.
\end{cor*}

It was previously unknown whether the polynomial Daugavet property transfers from the dual space to the space itself - a problem that has now been completely solved. In particular, \cite[Problem 4.4]{CMP24} has been positively answered, establishing that the polynomial Daugavet property of the predual of a JBW$^*$-triple (or even of a von Neumann algebra) is equivalent to that of the JBW$^*$-triple (or the von Neumann algebra) itself. This is so because the equivalence is already known for the linear DPr \cite[Theorem~3.2]{BM05}. 

Using similar methods, we also produce a Shvidkoy-like results for linear Daugavet centers - see Theorem \ref{theorem:Shvidkoy-for-Daugavetcenters} - and for the weak operator Daugavet property (WODP, for short) - see Theorem \ref{thm:polynomial_WODP_Shvidkoy}. We send the reader to the definitions of these two concepts in Sections \ref{sec:polynomial_DP} and \ref{sec:polynomial_WODP}, respectively, and recall that the WODP implies the DPr. As these properties have their counterparts for polynomials (which were formally stronger properties), what we get is the following.

\begin{cor*}[\mbox{\textrm Corollaries~\ref{Corollary-2.9} and \ref{cor:WODP-polynomialWODP}.(a)}]
\begin{samepage}
Let $X$ and $Y$ be Banach spaces. \parskip=0pt
    \begin{enumerate}
        \item[\textup{(a)}] If $G \in \mathcal{L}(X,Y)$ is a Daugavet center, then it is a polynomial Daugavet center.
        \item[\textup{(b)}] If $X$ has the WODP, then $X$ has the polynomial WODP. 
    \end{enumerate}   
\end{samepage}
\end{cor*}

As a consequence of the second statement of the last corollary and the result \cite[Theorem 5.12]{MRZ22}, we will get new examples of Banach spaces with the DPr. Indeed, we have the following result.

\begin{cor*}[\mbox{\textrm Corollary~\ref{cor:WODP-polynomialWODP}.(b)}]  
    Let $X$ be a Banach space with the WODP. Then, for every $N\in\N$, the $N$-fold symmetric tensor product $\widehat{\otimes}_{\pi,s,N}X$ of $X$ has the WODP, hence the DPr. In particular, the following Banach spaces $X$ are such that $\widehat{\otimes}_{\pi,s,N}X$ has the DPr. \parskip=0pt 
     \begin{enumerate}
\item[\textup{(b1)}] $L$-embedded spaces $X$ satisfying the MAP and the DPr, with $\dens(X) \leq \omega_1$.
\item[\textup{(b2)}] $X=W_1 \widehat{\otimes}_{\pi} W_2$, where $W_1$ and $W_2$ are $L_1$-preduals with the DPr, or $L_1(\mu, Y)$-spaces for $\mu$ atomless and $Y$ arbitrary, or spaces as in (a). 
    \end{enumerate} 
\end{cor*}

Let us finish this introduction by mentioning that we do not know whether there is a Bourgain-type lemma for the weak polynomial topology of the unit ball of a Banach space $X$, i.e.\ a result saying that every non-empty relatively open subset of $B_X$ for the weak polynomial topology contains a convex combination of slices. Such a result would immediately provide the Shvidkoy-type result that we get in this paper (Theorem~\ref{thm:polynomial_Shvidkoy_result}), as in the proof of \cite[Lemma~3]{Shvidkoy00} it is actually shown the following: {\slshape let $X$ be a Banach space with the DPr, let $U \subset B_X$ 
be a subset that contains a convex combination of slices of $B_X$, $x_0 \in S_X$, and $\eps > 0$; then there is $z \in U$ such that $\|x_0 +z\| > 2 - \eps$} (see \cite[Lemma 3.1.15]{KMRZW25} for a direct proof of this fact). But there are several indications that this kind of result might not always hold. On one hand, the weak polynomial topology is not always a linear topology (see \cite{GJL05} and references therein for more information), so it is possible that none of the standard geometric tools from topological vector spaces will be available for this topology. On the other hand, there are spaces where all convex combination of slices of the unit ball are relatively weakly open, as $c_0$ or have non-empty relative weak interior, as $L_1[0,1]$ (see \cite{LPMRZ19,HKP25} and references therein), and this could possibly conflict with the fact that the relative weak and weak polynomial topologies on the unit ball do not always coincide. It is worth noting also that our methods do not seem to be applicable to a possible proof that the DPr implies the WODP. Therefore, it seems that this problem is still open (this was explicitly stated after \cite[Remark 5.3]{MRZ22}).

\subsection{Outline}
The paper is structured as follows. We finish this introduction with a short subsection containing the notation and the needed definitions related to polynomials and the necessary background in order to avoid the reader jumping into specific references very often. In Section \ref{sec:mainresults} we provide our results. In Subsection~\ref{sec:polynomial_DP}, we prove our new version of Shvidkoy's lemma for the DPr valid for the weak polynomial topology, which provides the equivalence between the DPr and the polynomial DPr. In subsection \ref{sec:Daugavetcenter} we consider Daugavet centers and prove that every linear Daugavet center is a polynomial Daugavet center. In Subsection~\ref{sec:polynomial_WODP}, we establish a Shvidkoy-type lemma for the WODP valid for the weak polynomial topology which provides the equivalence between the WODP and its polynomial counterpart, and  we present an application of it to get new examples of Banach spaces for which the $N$-fold symmetric tensor product has the Daugavet property. It is worth noting that all the results in this paper hold for both real and complex Banach spaces.

\subsection{Notation, definitions, and preliminaries}\label{subsec:preliminaries}
Since this paper deals with polynomials on Banach spaces, it is necessary to introduce some additional machinery. Let $X, Y$ be Banach spaces over $\mathbb{K}$ and $N\in\N$. Recall that a mapping $P\colon X\to Y$ is called a continuous \emph{$N$-homogeneous polynomial} if we can find an $N$-linear bounded symmetric map $F\colon X^N\to Y$ (meaning that the following equality $F(x_{\sigma(1)},\dots,x_{\sigma(N)})=F(x_1,\dots, x_N)$ holds true for every permutation $\sigma$ of $\{1,\dots,N\}$ and for every $N$-tuple $(x_1,\dots,x_N) \in X^N$) such that $P(x)=F(x,\dots,x)$ for every $x\in X$. We denote by $\mathcal{P}(^N X, Y)$ the Banach space of all $N$-homogeneous continuous polynomials endowed with the supremum of the norm of the images of $B_X$ as (complete) norm. When $Y$ is the base field, we just write $\mathcal{P}(^N X)=\mathcal{P}(^N X,\K)$. A mapping $P\colon X\to Y$ is called a \emph{polynomial} if it is a finite linear combination of homogeneous continuous polynomials (we consider the constant mappings to be the $0$-homogeneous polynomials and we note that the bounded linear operators are the $1$-homogeneous continuous polynomials). Observe that we only consider in this paper continuous polynomials, so the word ``polynomial'' always suppose its continuity. We denote by $\Pol(X,Y)$ the set of all continuous polynomials from $X$ to $Y$, and by $\Pol(X)$ the set of all scalar-valued continuous polynomials on $X$. A polynomial $P\in \Pol(X,Y)$ is said to be \emph{weakly compact} if $P(B_X)$ is a relatively weakly compact subset of $Y$; $P$ is a \emph{rank-one} polynomial if $P(x)=p(x)y_0$ $\forall x\in X$, where $p\in \Pol(X)$ and $y_0\in Y$. We send the reader to \cite{Din99, Mujica86, HJ14} for a complete background on this topic.

Our main tool here will be the \emph{weak polynomial topology} on a Banach space $X$. It was introduced in \cite{CCG89} as the smallest topology on $X$ which makes all the scalar-valued polynomials on $X$ continuous. Equivalently, it is the smallest topology on $X$ for which a net $(x_\alpha)$ in $X$ converges to a point $x$ in $X$ if and only if $p(x_\alpha)\to p(x)$ for every $p\in\Pol(X)$. Finally, let us recall that the \emph{polynomial-star topology} of $X^{**}$ is the topology defined in \cite{DG89} as the smallest topology on $X^{**}$ for which a net $(x_\alpha)$ in $X^{**}$ converges to a point $x$ in $X^{**}$ if and only if $\widehat{p}(x_\alpha)\to \widehat{p}(x)$ for every scalar-valued polynomial $p$ on $X$, where $\widehat{p}$ denotes the Aron--Berner extension of $p$ to $X^{**}$ \cite{AB78}. It is shown in \cite[Theorem 2]{DG89} that $B_{X^{**}}$ is equal to the polynomial-star closure of $B_X$ in $X^{**}$, providing a polynomial-star analogue to the classical Goldstine theorem.

\section{The main results} \label{sec:mainresults}

\subsection{The DPr and the polynomial DPr are equivalent} \label{sec:polynomial_DP}

We will now state and prove the main result of the paper. 

\begin{thm}\label{thm:polynomial_Shvidkoy_result}
    Let $X$ be a Banach space with the DPr. Then, for every $x\in S_X$ and $y\in B_X$, we can find a net $(y_\alpha) \subseteq B_X$ which converges to $y$ in the weak polynomial topology and such that $\norm{x+y_\alpha}\to 2$. 
\end{thm}

As previously mentioned, the proof of this result essentially consists in using Shvidkoy's original $\eps$-quasi-codirected point argument from Lemma~\ref{lemma:DP_Shvidkoy_characterization} at every step of the construction of the net in the proof of the  Davie--Gamelin's Goldstine-type result for the weak-star polynomial topology given in \cite[Theorem 2]{DG89}, to force this net to be almost quasi-codirected with any previously fixed point of the unit sphere. More precisely, given $x\in S_X$, $y\in B_Y$, and a finite family $\F$ of continuous symmetric multilinear  forms on $X$, we will construct a finite sequence $y_1,\dots,y_N \in B_X$ as in the proof of \cite[Theorem 2]{DG89} whose arithmetic mean $y_\F$ will have value close to the value of $y$ under the action of every homogeneous polynomial arising from a multilinear form in $\F$ and, additionally, which will be almost quasi-codirected with the point $x$.  

We will also need the following well known immediate consequence of the
triangle inequality. 

\begin{lem}\label{lem:l_1^2_identity}
  Let $X$ be a Banach space, $\eps\in(0,1)$, and $x,y\in B_X$ be such that $\norm{x+y}> 2-\eps$. Then, for every $\lambda\in(0,1)$, we have \begin{equation*}
        \norm{x+\lambda y}> 1+\lambda-\eps.
    \end{equation*}
\end{lem}

We also need the following lemma, which we will prove after we present a proof for Theorem \ref{thm:polynomial_Shvidkoy_result}.

\begin{lem} \label{lemma-referee} Let $X$ be a Banach space with the DPr, let $x \in S_X$, let $y \in B_X$, let $\eps \in (0,1)$, let $\mathcal{F}$ be a finite family of continuous symmetric multilinear forms on $X$ and let $N \in \N$ be such that every $F \in \mathcal{F}$ is $m$-linear for some $m \leq N$. Then, we can find $y_1, \ldots, y_N \in B_X$ such that 
\begin{equation} \label{referee-lemma-1}
    \left\| x + \sum_{i=1}^N y_i \right\| > N + 1 - \frac{\eps}{N+1}
\end{equation} 
and 
\begin{equation} \label{referee-lemma-2}
    |F(y_{i_1}, \ldots, y_{i_m}) - F(y,\ldots, y)| < \eps
\end{equation}
for every $m$-linear form $F \in \mathcal{F}$ and every $i_1 < \ldots < i_m \in \{1, \ldots, N\}$.
\end{lem}

With Lemma \ref{lemma-referee} in mind, Theorem \ref{thm:polynomial_Shvidkoy_result} follows.
 
\begin{proof}[Proof of Theorem \ref{thm:polynomial_Shvidkoy_result}] Fix $x \in S_X$ and $y \in B_X$. Let $\eps \in (0,1)$ and let $\mathcal{F}$ be a finite family of continuous symmetric multilinear forms on $X$. Choose $N \in \N$ sufficiently large so that every $F \in \mathcal{F}$ is $m$-linear for some $m \leq N$ and so that the averaging in \cite[page 354]{DG89} applies. By Lemma \ref{lemma-referee}, there are $y_1, \ldots, y_N \in B_X$ satisfying (\ref{referee-lemma-1}) and (\ref{referee-lemma-2}) above. Set 
\begin{equation*}
    y_{\mathcal{F}}:= \frac{1}{N} \sum_{i=1}^N y_i.
\end{equation*}
On the one hand, it is proved in \cite[page 354]{DG89} that, for $N$ sufficiently large, (\ref{referee-lemma-2}) implies 
\begin{equation*}
    |F(y_{\mathcal{F}}, \ldots, y_{\mathcal{F}}) - F(y, \ldots, y)| < 2 \eps 
\end{equation*}
for every $F \in \mathcal{F}$. On the other hand, (\ref{referee-lemma-1}) and the Hahn-Banach theorem provide a functional $f \in S_{X^*}$ such that 
\begin{equation*}
    \re f(x) > 1 - \frac{\eps}{N+1} \ \ \ \mbox{and} \ \ \ \re f(x_i) > 1 - \frac{\eps}{N+1} 
\end{equation*}
for every $i \in \{1, \ldots, N\}$. Thus, we have that 
\begin{equation*}
    \|x + y_{\mathcal{F}}\| \geq \re f(x) + \frac{1}{N} \sum_{i=1}^N \re f(x_i) > 2 - \frac{2\eps}{N+1} \geq 2 - \eps. 
\end{equation*}
Taking as a direct set the product of $(0,1)$, ordered by reverse inequality, and the set of all finite families of continuous symmetric multilinear forms on $X$, ordered by inclusion, the preceding estimates produce a net $(y_{\alpha})$ such that $y_{\alpha}$ converges to $y$ in the weak polynomial topology and $\|x + y_{\alpha}\| \rightarrow 2$.
\end{proof} 

It remains to prove Lemma \ref{lemma-referee}.

\begin{proof}[Proof of Lemma \ref{lemma-referee}] As in \cite[Section 3]{DG89}, we select the points $y_i$, inductively. Put 
\begin{equation*}
    \xi:= \frac{\eps}{N(N+1)}.
\end{equation*}
We will construct $y_1, \ldots, y_N \in B_X$ such that, for every $n \in \{1, \ldots, N\}$,
\begin{equation}\label{referee-lemma-3}
       \norm{x+\sum_{i=1}^n y_i}> n+1- n \xi
   \end{equation} 
   and, additionally,
   \begin{equation}\label{referee-lemma-4}
       |F(y_{i_1}, \ldots, y_{i_{k-1}}, y_{i_k}, y, \ldots, y) - F(y_{i_1}, \ldots, y_{i_{k-1}}, y, y, \ldots, y)| < \frac{\eps}{N}
   \end{equation} 
for every $m$-linear symmetric form $F\in\F$, every $k \leq m$ and for every $i_1<\dots<i_k \leq n$. The case $n=1$ follows from Shvidkoy's lemma. Indeed, there is a net $(y_{\alpha})$ in $B_X$ which converges weakly to $y$ and satisfies $\|x+y_{\alpha}\| \rightarrow 2$. For every $F \in \mathcal{F}$, the mapping $z \mapsto F(z, y, \ldots, y)$ is linear and continuous, hence weakly continuous. Since $\mathcal{F}$ is finite, we can choose $\alpha$ such that $\|x+y_{\alpha}\|>2-\xi$ and 
\begin{equation*}
    |F(y_{\alpha}, y, \ldots, y) - F(y, \ldots, y)| < \frac{\eps}{N}
\end{equation*}
for every $F \in \mathcal{F}$. Then $y_1 := y_{\alpha}$ satisfies the required conditions. Now, suppose that $y_1, \ldots, y_n$ have been constructed for some $n \in \{1, \ldots, N-1\}$. Applying Shvidkoy's lemma (see once again Lemma~\ref{lemma:DP_Shvidkoy_characterization}) to the point $\frac{x+\sum_{i=1}^ny_i}{\norm{x+\sum_{i=1}^ny_i}} \in S_X$, we obtain a net $(y_\alpha)$ in $B_X$ which converges weakly to $y$ and such that \begin{equation*}
       \norm{\frac{x+\sum_{i=1}^ny_i}{\norm{x+\sum_{i=1}^ny_i}}+y_\alpha}\to 2.
   \end{equation*}
As before, we have that given any $m$-linear $F\in\F$, $k\leq m-1$ and $i_1<\dots<i_k\leq n$, the mapping $z\mapsto F(y_{i_1},\dots,y_{i_k},z,y,\dots,y)$ is linear and continuous, so \begin{equation*}
       F(y_{i_1},\dots,y_{i_k},y_\alpha,y,\dots,y)\to F(y_{i_1},\dots,y_{i_k},y,y,\dots,y).
   \end{equation*} 
   Since there is only a finite number of linear forms and finitely many possible choices for the multi-indices $i_1<\dots<i_k$ for each of these, we can find $\alpha$ such that all the required estimates in (\ref{referee-lemma-4}) hold, that is, 
   \begin{equation*}
       \bigl|F(y_{i_1},\dots,y_{i_k},y_\alpha,y,\dots,y)- F(y_{i_1},\dots,y_{i_k},y,y,\dots,y)\bigr|<\frac{\eps}{N}
   \end{equation*} 
   for every such choice, and additionally such that
   \begin{equation*}
       \norm{\frac{x+\sum_{i=1}^ny_i}{\norm{x+\sum_{i=1}^ny_i}}+y_\alpha}> 2-\frac{\xi}{\norm{x+\sum_{i=1}^ny_i}}.
   \end{equation*} 
   By using the induction hypothesis and applying Lemma~\ref{lem:l_1^2_identity} with $\lambda = \frac{1}{\norm{x+\sum_{i=1}^n y_i}} \in (0,1)$ (indeed, notice that $\norm{x+\sum_{i=1}^n y_i} > n$), we immediately get 
   \begin{equation*}
       \norm{x+\sum_{i=1}^ny_i+y_\alpha}> 1+\norm{x+\sum_{i=1}^ny_i}-\xi> n+2-(n+1)\xi.
   \end{equation*} So $y_{n+1}:=y_\alpha$ is as we wanted.
\vspace{0.2cm}
   Finally, note that, on the one hand, by (\ref{referee-lemma-3}) with $n=N$, we have 
   \begin{equation*}
       \norm{x+\sum_{i=1}^Ny_i}> N+1-N\xi=N+1-\frac{\eps}{N+1}
   \end{equation*} 
   which gives (\ref{referee-lemma-1}), and, on the other hand, given any $m$-linear symmetric form $F\in\F$ and $i_1<\dots<i_m\in\{1,\dots,N\}$, we have 
   \begin{align*}
       \bigl| F(y_{i_1},\dots,y_{i_m})-F(y,\dots,y)\bigr| &\leq \bigl| F(y_{i_1},\dots,y_{i_m})-F(y_{i_1},\dots,y_{i_{m-1}},y)\bigr| \\
       &+ \bigl|F(y_{i_1},\dots,y_{i_{m-1}},y)-F(y_{i_1},\dots,y_{i_{m-2}},y,y)\bigr| \\
       &+ \dots +\bigl|F(y_{i_1},y,\dots,y)-F(y,\dots,y)\bigr| \\
       &<m\cdot\frac{\eps}{N}\leq \eps. 
   \end{align*} This proves (\ref{referee-lemma-2}) and completes the proof.
\end{proof}

As a corollary, we immediately get that the DPr and the polynomial DPr coincide solving thus a longstanding open problem. In order to do this, we will use the characterization given in Proposition \ref{prop:char-pol-DPr}.

\begin{cor} \label{DPr-polynomialDPr}
    A Banach $X$ with the DPr satisfies the polynomial DPr.
\end{cor}

\begin{proof}
    Let $x\in S_X$, $p\in\Pol(X)$ with $\norm{p}=1$ and $\eps>0$. Pick $y\in B_X$ and a modulus-one scalar $\omega\in\K$ to be such that $\real\omega p(y)>1-\eps$. Applying Theorem~\ref{thm:polynomial_Shvidkoy_result} to the points $\overline{\omega}\, x$ and $y$, we can find a net $(y_\alpha)$ in $B_X$ which converges to $y$ in the weak polynomial topology and such that $\norm{x+\omega y_\alpha}\to 2$. Since, in particular, $p(y_\alpha)\to p(y)$, we can find $\alpha$ such that $\real\omega p(y_\alpha)>1-\eps$ and $\norm{x+\omega y_\alpha}>2-\eps$.
\end{proof}

\begin{rem}\slshape
    In particular, this result shows that for a \textit{real} Banach space with the DPr we have that $\norm{p}=\sup_{x\in S_X}\abs{p(x)}$ for every $p\in\Pol(X)$. In the complex case, this actually happens in all Banach spaces by the Maximum Modulus Principle (see, for instance, \cite[5.G, page 40]{Mujica86}) but, in the real case, it is not always the case. Indeed, just consider the polynomial $x\mapsto 1-\scal{x,x}$ defined on a Hilbert space.
\end{rem}

\subsection{Every Daugavet center is a polynomial Daugavet center} \label{sec:Daugavetcenter}
Let us now deal with an extension of the Daugavet property, the Daugavet centers introduced in \cite{BK10}.

Let $X$, $Y$ be Banach spaces. A (non-null) operator $G\in \mathcal{L}(X,Y)$ is called a \emph{Daugavet center} \cite{BK10} if the equation 
\begin{equation*}
    \norm{G+T}=\norm{G}+\norm{T}
\end{equation*} 
holds for every rank-one $T\in\mathcal{L}(X,Y)$. In this case, the same norm-equality is satisfied by all bounded linear operators not fixing a copy of $\ell_1$ (in particular, by all weakly compact operators) \cite[Corollary 3]{BK10}. Following \cite{Santos20}, we say that a (non-null) polynomial $Q\in \Pol(X,Y)$ is a \emph{polynomial Daugavet center} if the equation
\begin{equation*}
    \norm{Q+P}=\norm{Q}+\norm{P}
\end{equation*} 
holds for every rank-one polynomial $P\in\Pol(X,Y)$. In this case, all weakly compact $P\in\Pol(X,Y)$ also satisfy the same norm equality (see \cite[Corollary 2.9]{Santos20}). Note that a Banach space $X$ has the DPr (respectively, polynomial DPr) if and only if the identity $\Id$ is a Daugavet (respectively, polynomial Daugavet) center. Also, by homogeneity, $G$ is a Daugavet center (respectively, $Q$ is a polynomial Daugavet center) if and only if  $\frac{G}{\norm{G}} \in S_{\mathcal{L}(X,Y)}$ is a Daugavet center (respectively, $\frac{Q}{\norm{Q}} \in S_{\mathcal{P}(X,Y)}$ is a polynomial Daugavet center). Thus, when dealing with these properties, it is enough to work with norm-one operators or polynomials. For these notions, we have the following geometric characterizations.

\begin{prop}\label{prop:characterizations-DC-PolDC}
    Let $X$, $Y$ be Banach spaces, let $G\in\mathcal{L}(X,Y)$ with $\|G\|=1$, and let $Q\in \Pol(X,Y)$ with $\|Q\|=1$. Then we have the following statements.\parskip=0pt
    \begin{enumerate}
        \item[\textup{(a)}] \cite[Theorem~2.1]{BK10} The operator $G$ is a Daugavet center if and only if it satisfies the following condition: for every $y\in S_Y$ and every slice $S$ of $B_X$, we have that $\sup_{x\in S}\norm{y+G(x)}=2$.
        \item[\textup{(b)}] \cite[Theorem~2.2]{Santos20} The polynomial $Q$ is a polynomial Daugavet center if and only if it satisfies the following condition: for every $y\in S_Y$, $\eps>0$, and $p\in \Pol(X)$ with $\|p\| = 1$, there exists $x\in B_X$ and a modulus-one scalar $\omega\in\K$ such that $\real \omega p(x)>1-\eps$ and $\norm{y+\omega Q(x)}>2-\eps$.
    \end{enumerate}
\end{prop}

The following Shvidkoy-type lemma for Daugavet center was proved in \cite{Bosenko10}. 

\begin{lem}[\mbox{\cite[Lemma~3]{Bosenko10}}]\label{lemma:Bosenko}
    Let $X$ be a Banach space and $G\in \mathcal{L}(X,Y)$ with $\|G\|=1$ be a Daugavet center. Then for every $x\in B_X$ and $y\in S_Y$, there exists a net $(x_\alpha)$ in $B_X$ such that $x_\alpha\to x$ weakly and $\norm{y+G(x_\alpha)}\to 2$.
\end{lem}

The following result is presented with a sketch of the proof. We leave the full details to the reader, indicating only the essential adjustments to the proof of Theorem~\ref{thm:polynomial_Shvidkoy_result}.

\begin{thm}\label{theorem:Shvidkoy-for-Daugavetcenters}
    Let $X$ be a Banach space and $G\in \mathcal{L}(X,Y)$ be a Daugavet center. Then for every $x\in B_X$ and $y\in S_Y$, there exists a net $(x_\alpha)$ in $B_X$ such that $x_\alpha\to x$ in the weak polynomial topology and $\norm{y + G(x_\alpha)}\to 2$.
\end{thm}

\begin{proof} The proof mirrors that of Theorem \ref{thm:polynomial_Shvidkoy_result} with the only difference being the treatment of the $\eps$-quasi-codirected argument. By homogeneity, we may assume that $\|G\|=1$. Fix $x \in B_X$ and $y \in S_Y$. Given $\eps \in (0,1)$, a finite family $\mathcal{F}$ of continuous symmetric multilinear forms on $X$ and $N \in \N$ sufficiently large so that every $F \in \mathcal{F}$ is $m$-linear for some $m \leq N$, we need to construct $x_1, \ldots, x_N \in B_X$ such that 
\begin{equation*}
    \left\| y + \sum_{i=1}^N G(x_i) \right\|_Y > N + 1 - \frac{\eps}{N+1}
\end{equation*}
and 
\begin{equation*}
    |F(x_{i_1}, \ldots, x_{i_m}) - F(x, \ldots, x)| < \eps 
\end{equation*}
for every $m$-linear $F \in \mathcal{F}$ and every $i_1 < \ldots < i_m \in \{1, \ldots, N\}$. We then use 
\begin{equation*}
    x_{\mathcal{F}}:= \frac{1}{N} \sum_{i=1}^N x_i
\end{equation*}
and conclude the proof as in the proof of Theorem \ref{thm:polynomial_Shvidkoy_result}. For the construction, put 
\begin{equation*}
    \xi := \frac{\eps}{N(N+1)}
\end{equation*}
and proceed by induction. The case $n=1$ is obtained exactly as in the proof of Lemma \ref{lemma-referee}. Suppose now that $x_1, \ldots, x_n$ have already been chosen and that 
\begin{equation*}
    \left\| y + \sum_{i=1}^n G(x_i) \right\|_Y > n + 1 - n \xi.
\end{equation*}
Lemma \ref{lemma:Bosenko} applied to $\frac{y+\sum_{i=1}^n G(x_i)}{\|y+\sum_{i=1}^n G(x_i)\|} \in S_Y$ and to the point $x \in B_X$ allows us to choose $x_{n+1} \in B_X$ satisfying the required multilinear estimates and also 
\begin{equation*}
    \left\| \frac{y+\sum_{i=1}^n G(x_i)}{\|y+\sum_{i=1}^n G(x_i)\|} + G(x_{n+1}) \right\|_Y > 2 - \frac{\xi}{\| y + \sum_{i=1}^n G(x_i)\|}.
\end{equation*}
We apply Lemma \ref{lem:l_1^2_identity} for $\lambda = \frac{1}{\|y + \sum_{i=1}^n G(x_i)\|} \in (0,1)$. It follows that 
\begin{equation*}
    \left\| y + \sum_{i=1}^{n+1} G(x_i) \right\|_Y > n+2 - (n+1)\xi.
\end{equation*}
This is the desired inductive inequality. The multilinear estimates are obtained simultaneously from weak continuity, exactly as in the proof of Lemma \ref{lemma-referee}. The rest of the proof follows as in Theorem \ref{thm:polynomial_Shvidkoy_result} using the linearity of $G$.
\end{proof}

Combining Theorem \ref{theorem:Shvidkoy-for-Daugavetcenters} with Proposition \ref{prop:characterizations-DC-PolDC}.(b), and applying the same reasoning as in the proof of Corollary \ref{DPr-polynomialDPr}, we obtain the following result for any linear Daugavet center $G \in \mathcal{L}(X,Y)$.

\begin{cor}\label{Corollary-2.9}
    Every (linear) Daugavet center is a polynomial Daugavet center.
\end{cor}

Let us observe that the linearity of the operator $G$ is crucial for adapting the proof of Theorem~\ref{thm:polynomial_Shvidkoy_result} to Daugavet centers (more precisely, to make the $\eps$-quasi-codirected point argument to work), and we do not know if our proof can be adapted in the non-linear setting. In particular, we do not know if, given a (non-null) polynomial $Q\in P(X,Y)$ which satisfies \begin{equation*}
    \norm{Q+T}=\norm{Q}+\norm{T}
\end{equation*} for every rank-one $T\in\mathcal{L}(X,Y)$, we actually have that $Q$ is a polynomial Daugavet center, that is, the same norm equality holds for $T$ being a rank-one polynomial from $X$ to $Y$. 

\subsection{The WODP and the polynomial WODP are equivalent}\label{sec:polynomial_WODP}
In this section, we address two properties introduced in \cite{MRZ22}, namely the WODP and the polynomial WODP, which were proposed to study the preservation of the Daugavet property under projective tensor products and projective symmetric tensor products. We need some additional notation. Given a Banach space $X$, $x_1,\dots,x_n\in B_X$, $\eps>0$, and $x'\in B_X$, we denote by $\OF(x_1,\dots,x_n;x',\eps)$ the set of all points $x\in B_X$ for which there exists $T\in\mathcal{L}(X,X)$ such that $\norm{T}\leq 1$, $\norm{T(x_i)-x_i}<\eps$ for every $i\in\{1,\dots,n\}$, and $\norm{T(x)-x'}<\eps$. 

\begin{defn}[\mbox{\cite{MRZ22}}]
    Let $X$ be a Banach space. We say that \parskip=0pt
    \begin{enumerate}
        \item[\textup{(a)}] $X$ has the \emph{weak operator Daugavet property} (\emph{WODP}, for short) if for every $x_1,\dots,x_n\in B_X$, $\eps>0$, $x'\in B_X$, and every slice $S$ of $B_X$, we have that $S\cap \OF(x_1,\dots,x_n;x',\eps)\neq\emptyset$;
        \item[\textup{(b)}] $X$ has the \emph{polynomial weak operator Daugavet property} (\emph{polynomial WODP}, for short) if for every $x_1,\dots,x_n\in B_X$, $\eps,\delta>0$, $x'\in B_X$, and $p\in\Pol(X)$ with $\|p\|=1$, we can find $y\in \OF(x_1,\dots,x_n;x',\eps)$ and a modulus-one scalar $\omega\in\K$ such that $\real \omega p(y)>1-\delta$. 
    \end{enumerate}
\end{defn}

\begin{rem}\slshape Let us notice that our definition for the set $\OF(x_1,\dots,x_n;x',\eps)$ differs slightly from the one in \cite{MRZ22}, but it is readily checked that the properties obtained in the above way are exactly the WODP and the polynomial WODP as defined in \cite{MRZ22}. See also \cite[page 6, immediately before Lemma 2.2]{RZ24} .
\end{rem}

The interest in the WODP lies in the fact that it is stable by taking projective tensor products \cite[Theorem~5.4]{MRZ22}, and that it has other interesting applications for the geometry of injective and projective tensor products \cite{RZ24}.
The analogue to Shvidkoy's result in this setting was proved in \cite[Lemma~5.6]{MRZ22}.

\begin{lem}[\mbox{\cite[Lemma~5.6]{MRZ22}}]\label{lemma:WODP_Shvidkoy}
    Let $X$ be a Banach space with the WODP. Then, for every $x_1,\dots,x_n\in B_X$, $\eps>0$, and $x'\in B_X$,  the set $\OF(x_1,\dots,x_n;x',\eps)$ is weakly dense in $B_X$.
\end{lem}

In the same vein as Theorems \ref{thm:polynomial_Shvidkoy_result} and \ref{theorem:Shvidkoy-for-Daugavetcenters}, we present a stronger version of Lemma \ref{lemma:WODP_Shvidkoy}, formulated once again for the weak polynomial topology. 

\begin{thm}\label{thm:polynomial_WODP_Shvidkoy}
    Let $X$ be a Banach space with the WODP. Then, given $x_1,\dots,x_n\in B_X$, $\eps>0$, and $x'\in B_X$,  the set $\OF(x_1,\dots,x_n;x',\eps)$ is dense in $B_X$ for weak polynomial topology.
\end{thm}

The proof follows similar arguments to those used in Theorem~\ref{thm:polynomial_Shvidkoy_result}. As in the proof of Theorem~\ref{theorem:Shvidkoy-for-Daugavetcenters}, the main difference lies in the need to find a suitable alternative to the $\eps$-quasi-codirected point technique. In order to achieve this, we essentially follow the ideas presented in \cite[Lemma~5.5]{MRZ22}.

\begin{proof}[Proof of Theorem~\ref{thm:polynomial_WODP_Shvidkoy}]
   Fix $x_1,\dots,x_\ell\in B_X$, $\eps>0$, and $x'\in B_X$. Then take $y\in B_X$. Again, it is sufficient to prove the following: given $\eps>0$, a finite family of $\F$ of continuous symmetric multilinear forms on $X$, and $N\in\N$ large enough so that every $F\in \F$ is $m$-linear for some $m\leq N$, we can find $y_1,\dots,y_N\in$ $S_X$ and $T\in \mathcal{L}(X,X)$ with $\norm{T}\leq 1$ such that \begin{equation}
   \label{eqt:operator_l_1_sequence}
       \norm{T(x_l)-x_l} <\eps \ \forall l\in\{1,\dots,\ell\},  \qquad \norm{T(y_i)-x'}<\eps\ \forall i\in\{1,\dots,N\};
   \end{equation}
   and \begin{equation*}
       \bigl| F(y_{i_1},\dots,y_{i_m})-F(y,\dots,y)\bigr|<\eps
   \end{equation*} for every $F\in\F$ $m$-linear and for every $i_1<\dots<i_m\in\{1,\dots,N\}$. Indeed, on one hand, the latter condition would give, as in the proof of Theorem~\ref{thm:polynomial_Shvidkoy_result}, that taking $N$ large enough and letting $y_\F:=\frac{1}{N}\sum_{i=1}^Ny_i$, we would get 
   \begin{equation*}
       \bigl|F(y_\F,\dots,y_\F)-F(y,\dots,y)\bigr|<2\eps
   \end{equation*} for every $F\in\F$. On the other hand, from \eqref{eqt:operator_l_1_sequence} and a straightforward application of the triangle inequality, we would get that \begin{equation*}
       \norm{T(y_\F)-x'}\leq\frac{1}{N}\sum_{i=1}^N\norm{T(y_i)-x'}<\eps.
   \end{equation*} Therefore, $y_\F\in \OF(x_1,\dots,x_\ell;x',\eps)$, and we would immediately get the net converging to $y$ we wanted.

Once again, we will construct the $y_n$'s inductively. Fix $\eps>0$, and let $\xi:=\frac{\eps}{N}$. We will construct $y_1,\dots,y_N$ in $B_X$ and $T_1,\dots, T_N\in \mathcal{L}(X,X)$ such that for every $n\in\{1,\dots,N\}$, we have $\norm{T_n}\leq 1$;
   \begin{equation*}
      \norm{T_n(x_l)-x_l}<n\xi \ \forall l\in\{1,\dots,\ell\} \ \ \text{ and } \ \     \norm{T_n(x')-x'}<n\xi; 
      \end{equation*} 
      \begin{equation*}\norm{T_n(y_i)-x'}<n\xi \ \forall i\in\{1,\dots,n\};
   \end{equation*} 
   and 
   \begin{equation*}
       \bigl|F(y_{i_1},\dots,y_{i_{k-1}},y_{i_k},y,\dots,y)- F(y_{i_1},\dots,y_{i_{k-1}},y,y,\dots,y)\bigr|<\frac{\eps}{N}
   \end{equation*} 
   for every $m$-linear symmetric form $F\in\F$, $k\leq m$, and $i_1<\dots<i_k\leq n$.

  The case $n=1$ immediately follows from Lemma~\ref{lemma:WODP_Shvidkoy} by combining the weak continuity of the mapping $z\mapsto F(z,y,\dots,y)$ and the weak density of the set $\OF(x',x_1,\dots,x_\ell;x',\xi)$, where we also ask the point $x'$ itself to be almost fixed by our operator. 

   Now, assume that $y_1,\dots,y_n$ and $T_1,\dots,T_n$ satisfy the above conditions for some $n\in\{1,\dots,N-1\}$. Applying Lemma~\ref{lemma:WODP_Shvidkoy} with an additional constraint on the values of the operator at the points $y_1,\dots,y_n$, and using the weak continuity of the (finite number of) mappings $z\mapsto F(y_{i_1},\dots,y_{i_k},z,y,\dots,y)$, we obtain a point $y_{n+1}$ in $\OF(x',x_1,\dots,x_\ell,y_1,\dots,y_n;x',\xi)$ such that given any $m$-linear symmetric form $F\in\F$, $k\leq m-1$, and $i_1<\dots<i_k\leq n$, we have 
   \begin{equation*}
       \bigl| F(y_{i_1},\dots,y_{i_k},y_{n+1},y,\dots,y)- F(y_{i_1},\dots,y_{i_k},y,y,\dots,y)\bigr|<\frac{\eps}{N}.
   \end{equation*} 
   Let $S\in \mathcal{L}(X,X)$ with $\|S\|=1$ witnessing the fact that $y_{n+1}$ belongs to 
   \begin{equation*} \OF(x',x_1,\dots,x_\ell,y_1,\dots,y_n;x',\xi).
   \end{equation*} We claim that the operator $T_{n+1}:=T_n\circ S$ does what we want. Indeed, first, for every $l\in\{1\dots,\ell\}$, we have \begin{align*}
   \norm{T_{n+1}(x_l)-x_l}&\leq \norm{T_n(S(x_l)-x_l)}+\norm{T_n(x_l)-x_l} \\ &< \norm{S(x_l)-x_l}+n\xi< (n+1)\xi.
   \end{align*} 
   Similarly, 
   \begin{align*}
   \norm{T_{n+1}(x')-x'} &\leq \norm{T_n(S(x')-x')}+\norm{T_n(x')-x'} < (n+1)\xi.
   \end{align*} 
   Second, for every $i\in\{1,\dots,n\}$, we have 
   \begin{align*}
   \norm{T_{n+1}(y_i)-x'}&\leq \norm{T_n(S(y_i)-y_i)}+\norm{T_n(y_i)-x'} \\ &< \norm{S(y_i)-y_i}+n\xi< (n+1)\xi.
   \end{align*} 
   Last,  
   \begin{equation*}
   \norm{T_{n+1}(y_{n+1})-x'}\leq \norm{T_n(S(y_{n+1})-x')}+\norm{T_n(x')-x'}< (n+1)\xi.
   \end{equation*} So we can now conclude as in the proof of Theorem~\ref{thm:polynomial_Shvidkoy_result}. 
\end{proof}

We conclude this work with an application of Theorem~\ref{thm:polynomial_WODP_Shvidkoy}. Given a Banach space $X$ and $N \in \mathbb{N}$, the ($N$-fold) {\it projective symmetric tensor product} of $X$, denoted by $\widehat{\bigotimes}_{\pi,s,N}X$, is the completion of the space $\bigotimes^{s,N} X$ endowed with the norm 
\begin{equation*}
    \|u\|_{\pi, s, N} := \inf \left\{ \sum_{i=1}^n |\lambda_i| \|x_i\|^N: \ u = \sum_{i=1}^n \lambda_i x_i^N, \ n \in \mathbb{N}, \ x_i \in X \right\}, 
\end{equation*}
where the infimum is taken over all such representations of $u$ and $x^N := x \otimes \overset{N}{\ldots} \otimes x$ for every $x \in X$. It turns out that the topological dual $(\widehat{\bigotimes}_{\pi,s,N}X)^*$ can be identified (there is an isometric isomorphism) with $\mathcal{P}(^N X)$. We refer the interested reader to \cite{BR01, CD00, CG11, Floret97} for further background on the topic.

As a consequence of Theorem \ref{thm:polynomial_WODP_Shvidkoy}, we get the following consequence. 

\begin{cor} \label{cor:WODP-polynomialWODP} Let $X$ be a Banach space and $N \in \mathbb{N}$. Suppose that $X$ has the WODP. Then,\parskip=0pt
\begin{enumerate}
\item[\textup{(a)}] $X$ satisfies the polynomial WODP.
\item[\textup{(b)}] $\widehat{\bigotimes}_{\pi,s,N}X$ has the WODP.
\end{enumerate}
In particular, the following Banach spaces $X$ are such that $\widehat{\otimes}_{\pi,s,N}X$ has the DPr.  
     \begin{enumerate}
\item[\textup{(b1)}] $L$-embedded spaces $X$ satisfying the MAP and the DPr, with $\dens(X) \leq \omega_1$.
\item[\textup{(b2)}]  $X=W_1 \widehat{\otimes}_{\pi} W_2$, where $W_1$ and $W_2$ are $L_1$-preduals with the DPr, or $L_1(\mu, Y)$-spaces for $\mu$ atomless and $Y$ arbitrary, or spaces as in (b1). 
    \end{enumerate}     
\end{cor}

\begin{proof} We first prove (a). Fix $x_1, \ldots, x_n \in B_X$, $\eps, \delta > 0$, $x' \in B_X$ and $p \in \mathcal{P}(X)$ with $\|p\| = 1$. Choose $y \in B_X$ and a modulus one scalar $\omega \in \K$ such that $$\re \omega p(y) > 1 - \delta/2.$$ By Theorem \ref{thm:polynomial_WODP_Shvidkoy}, there exists $z \in \OF(x_1, \ldots, x_n; x', \eps)$ such that  $|p(z) - p(y)| < \delta/2$. Consequently, $\re \omega p(z) \geq \re \omega p(y) - |p(z) - p(y)| > 1 - \delta$, which proves (a). That (a) implies (b) was proved in \cite[Theorem 5.12]{DJRZ22}. It was proved in \cite[Theorem 3.1]{DJRZ22} that the hypothesis of (b1) implies that $X$ has the WODP, so (b) gives the result. By \cite[Propositions 5.9 and 5.11]{MRZ22} and \cite[Theorem~3.1]{DJRZ22}, each $W_1$ and $W_2$ in (b2) has the WODP. Therefore, \cite[Theorem 5.4]{MRZ22} implies that $X = W_1 \widehat{\otimes}_{\pi} W_2$ has the WODP as well. Now, (b) gives the result.
\end{proof}

\section*{Acknowledgments}
Part of this work was carried out during the visit of the third author to the University of Granada in June 2025, following his participation in the BIRS-IMAG conference {\it Functional and Metric Analysis and their Interactions}, which was also attended by the other two authors. He gratefully acknowledges the warm hospitality extended by his colleagues and wishes to thank all those who contributed to making his visit possible. The authors would like to thank Mingu Jung for highlighting some details in the main proof that contributed to the final improved version of the paper. We also thank the anonymous referee for the detailed revision of the manuscript and for multiple stylistic suggestions which have improve the final version of paper.  

\section*{Funding}
Sheldon Dantas has been supported by the grants PID2021-122126NB-C31 and PID2021-122126NB-C33 funded by MICIU/AEI/10.13039/ 501100011033 and by ERDF/EU.
Miguel Mart\'{\i}n has been supported by the grant PID2021-122126NB-C31 funded by MICIU/AEI/10.13039/501100011033 and ERDF/EU, grant ``Mar\'{\i}a de Maeztu'' Excellence Unit IMAG funded by MICIU/AEI/10.13039/501100011033 with reference CEX2020-001105-M, and by grant FQM-0185 funded by Junta de Andaluc\'{\i}a. Yo\"el~Perreau has been supported by the Estonian Research Council grant PRG2545.

\end{document}